\documentclass[12pt]{amsart}
\usepackage{amssymb, latexsym, amsmath, amsfonts, hyperref, enumitem}

\newtheorem{thm}{Theorem}[section]
\newtheorem{cor}[thm]{Corollary}

\newtheorem{prop}[thm]{Proposition}
\theoremstyle{definition}

\theoremstyle{remark}
\newtheorem{rem}[thm]{Remark}

\numberwithin{equation}{section}

\newcommand{\no}{\noindent}

\begin{document}
\title{Fatou-Julia dichotomy of matrix-valued polynomials}
\thanks{The author is  supported by National Post-doctoral Fellowship, SERB, India}
\author{Ratna Pal}

\address{Department of Mathematics, Indian Institute of Science Education and Research, Pune, Maharashtra-411008, India}
\email{ratna.math@gmail.com} 

\begin{abstract}
This article gives a precise description of the Fatou sets and Julia sets  of matrix-valued polynomials in $\mathcal{M}(2,\mathbb{C})$ in terms of the corresponding polynomials in $\mathbb{C}$. Further,  we construct Green functions and B\"{o}ttcher-type functions for these matrix-valued polynomials.
	
\end{abstract}

\maketitle
\no 

\section{Introduction}
\no 
The map $\varphi:z\mapsto z^2$ is a central example in complex dynamics. Its
iterates can be computed and the Fatou-Julia dichotomy  can
be seen in a transparent manner. Motivated by this example, we look at $\mathcal{M}(2,\mathbb{C})$, the set of all $2$ by $2$ matrices over $\mathbb{C}$, and consider
the map $M \mapsto M^2$. This is in fact a polynomial map of the form
\[
(x,y,z,t) \mapsto (x^2+yz, y(x+t),  z(x+t),  t^2+yz)
\] 
from $\mathbb{C}^4$ to itself with maximal rank.
The dynamics of such matrix-valued maps has recently been looked
at in \cite{CD}. 
By working within the paradigm of matrix-valued maps,  we hope to study the dynamics of a large class of holomorphic endomorphisms in $\mathbb{C}^4$.
The underlying ring structure of matrices plays a key role in this direction. 

Let $\Phi: M\mapsto M^2$.  For $M\in \mathcal{M}(2,\mathbb{C})$, let $\rho(M)$ be the spectral radius of $M$.  Just as the unit circle plays an important role in determining the dynamics of the map $z\mapsto z^2$ in $\mathbb{C}$,  it turns out that a similar role is played by  
$$
J_\Phi=\{M: \rho(M)=1\},
$$
 which is in fact  the Julia set of $\Phi$.
This similarity is not merely a coincidence, rather gives a hint towards understanding the Julia set of a matrix-valued polynomial $P$ in terms of the corresponding polynomial $p$ in $\mathbb{C}$.

In this note, we study the Fatou-Julia dichotomy of the matrix-valued polynomials $\Phi$ in $\mathcal{M}(2,\mathbb{C})$ which respect conjugacy, i.e,
\[
{\Phi(QMQ^{-1})=Q\Phi(M)Q^{-1}}
\]
for all $Q\in {\rm{GL}}(2,\mathbb{C})$, the set of all invertible matrices in $\mathcal{M}(2,\mathbb{C})$. It turns out that such matrix-valued polynomials are precisely of the form
\begin{equation}\label{poly}
P(M)=a_k M^k+ a_{k-1}M^{k-1}+\cdots +a_1 M+a_0
\end{equation}
for $M\in \mathcal{M}(2,\mathbb{C})$ where $a_i\in \mathbb{C}$ for $1\leq i \leq k$ with $k\geq 2$ and $a_k\neq 0$. Therefore, there is a one to one correspondence between the set of polynomials in $\mathbb{C}$ and the set of matrix-valued polynomials in $\mathcal{M}(2,\mathbb{C})$ which respect conjugacy.

We denote the set of all diagonal matrices and the set of all non-diagonal matrices in $\mathcal{M}(2,\mathbb{C})$ by $\mathcal{D}$ and $\mathcal{ND}$ respectively. Further, $[\lambda_1,\lambda_2]$ and $[\lambda]$ will denote the matrices 
\[
\left[ {\begin{array}{cc}
\lambda_1 & 0 \\
0 & \lambda_2 \\
\end{array} } \right]
\text{ and }
\left[ {\begin{array}{cc}
\lambda & 1 \\
0 & \lambda \\
\end{array} } \right]
\]
respectively. Let 
\begin{equation*}
K_p=\{z: \{p^r(z)\}_{r\geq 1} \text{ is bounded}\},\ J_p=\partial K_p,\ F_\infty(p)=\mathbb{C}\setminus K_p
\end{equation*}
and
\begin{align*}
 K_P=\{M\in \mathcal{M}(2,\mathbb{C}): \{P^r(M)\}_{r\geq 1} \text{ is bounded}\}.
\end{align*}
The set $F_\infty(P)$ denotes the collection of all matrices in $\mathcal{M}(2,\mathbb{C})$ with at least one of the eigenvalues in $F_\infty(p)$. Note that $F_\infty(P)$ is always open in $\mathbb{C}^4$ but the set $K_P$ is not necessarily closed. In fact, we shall see in Section 2 that $F_\infty(P)$ is the complement of $\overline{K_P}$. We record the main theorem.
\begin{thm}\label{Julia}
Let $p$ be a polynomial in $\mathbb{C}$ with $\deg(p)\geq 2$ and $P$ be the corresponding matrix-valued polynomial. Then the Julia set of $P$ is 
\begin{align*}
J_P=\partial\overline{K_P}=\{Q[\lambda_1,\lambda_2]Q^{-1}: (\lambda_1,\lambda_2)\in (K_p\times J_p)\cup (J_p\times K_p), Q\in {\rm{GL}}(2,\mathbb{C})\}\\
\bigcup \{Q[\lambda]Q^{-1}:\lambda\in J_p, Q\in {\rm{GL}}(2,\mathbb{C})\}.
\end{align*}
\end{thm}
Hence, the Julia set $J_P$ of $P$ can be completely characterized in terms of the Julia set $J_p$ of $p$ and as a consequence, we obtain many interesting properties of the Fatou set and Julia set of $P$. For example, a direct application of Theorem \ref{Julia} shows that there is no wandering Fatou component of $P$ (see \cite{MXD}, for an example of a two dimensional polynomial map with wandering Fatou component). Also, the Fatou set and Julia set of $P$ are completely invariant under $P$ which is again not the case in general (see \cite{FS1}).

The first step in proving Theorem \ref{Julia} is to look at the Jordan canonical forms of the matrices in $\mathcal{M}(2,\mathbb{C})$. Since any of the eigenvalues of a matrix $M\in \mathcal{M}(2,\mathbb{C})$ being in $F_\infty(p)$ assures the uniform divergence of the sequence $\{P^n\}$ to infinity in a small neighborhood of $M$, the eigenvalues of the matrices which lie in the Julia set $J_P$ must be in the filled Julia set $K_p$ of the polynomial $p$. 
Next, we show that $J_P\cap {\rm{int}}(K_P)=\emptyset$. We shall see in Section 2 that ${\rm{int}}(K_P)$ consists precisely of those matrices whose eigenvalues lie in ${\rm{int}}(K_p)$ and by the classification theorem of Fatou components  (see \cite{MNTU}), any such eigenvalue lies either in an  attracting Fatou component or in a parabolic Fatou component or in a Siegel disk. Then we apply Theorem $3.1$ in \cite{AK} to show the normality of the sequence $\{P^n\}$ in a small neighborhood of any matrix in ${\rm{int}}(K_P)$. Hence, finally we conclude that the Julia set $J_P$ of $P$ consists of those matrices in $\mathcal{M}(2,\mathbb{C})$ whose at least one of the eigenvalues is in $J_p$ and the other is in $K_p$.

Next, we construct Green functions of matrix-valued polynomials $P$ in $\mathbb{C}^4$. For each $n\geq 1$, let 
\[
G_n(M)=\frac{1}{d^n}\log \lVert P^n(M)\rVert
\]
for $M\in \mathcal{M}(2,\mathbb{C})$. Clearly, each $G_n$ is a pluri-subharmonic function in $\mathbb{C}^4$.
\begin{thm}\label{Green fun}
Let $p$ be a polynomial in $\mathbb{C}$ of degree $d\geq 2$ and $P$ be the corresponding matrix-valued polynomial in $\mathbb{C}^4$. Then the sequence of functions $G_n$ converges uniformly on compacts to the function $G$ in $\mathbb{C}^4$. The function $G$ is continuous and pluri-subharmonic in $\mathbb{C}^4$ and pluri-harmonic in ${\rm{int}}(K_P) \cup F_\infty(P)$ vanishing identically on $\overline{K_P}$. Further, 
\[
G\circ P(M)=d G(M)
\]
for all $M\in \mathcal{M}(2,\mathbb{C})$ and
\[
G(M)=\log \rho(M)+O(1)
\]  
in $\{M\in \mathcal{M}(2,\mathbb{C}):\rho(M)>R\}$ for $R$ sufficiently large.
\end{thm}
 Let $R>1$ be sufficiently large such that the B\"{o}ttcher function $\varphi_p$ of $p$ is well-defined in $\Omega_R=\{z\in \mathbb{C}:\lvert z\rvert >R\}$.  Now consider the open set 
$$
\Omega=\{M\in \mathcal{M}(2,\mathbb{C}): \text{the eigenvalues of $M$ lie in } \Omega_R\}.
$$
We construct a function 
$\Phi_P: \Omega \rightarrow \Omega$, 
analogous to the B\"{o}ttcher function $\varphi_p$ of $p$ in $\mathbb{C}$. 
\begin{thm}\label{Bot}
Let $p$ be a polynomial in $\mathbb{C}$ of degree $d\geq 2$ and $P$ be the corresponding matrix-valued polynomial. Then, there exists a holomorphic function $\Phi_P: \Omega\rightarrow \Omega$ such that
\begin{equation*} \label{pro 1}
\Phi_P \circ P(M)= {(\Phi_P(M))}^d
\end{equation*}
and 
\begin{equation*}\label{pro 2}
G(M)=\log \rho(\Phi_P(M))
\end{equation*}
for all $M\in \Omega$.
\end{thm}
The results obtained in this note show that the dynamics of a matrix-valued polynomial $P$ in $\mathcal{M}(2,\mathbb{C})$ to a large extent replicates the dynamics of the corresponding polynomial $p$ in $\mathbb{C}$.  We belive that implementing a similar line of arguments, the results obtained here can be generalized for matrix-valued polynomials in $\mathcal{M}(n,\mathbb{C})$ for $n\geq 3$. Further, it would be also interesting to investigate the dynamics of the matrix-valued polynomials which are not necessarily compatible with conjugation. The dynamics of one such example, namely the map $\Phi:M\rightarrow [\lambda, \lambda^{-1}]M^2$, has been considered in \cite{CD}. Due to lack of compatibility with conjugation, the dynamics of the maps of this kind is expected to be more complicated than the case we have considered in this note and thus  a precise description of the Julia sets might not always be possible. 

\medskip 
\no 
\no {\bf Acknowledgement:} I would like to thank Kaushal Verma for  valuable  suggestions.



\section{Matrices having bounded orbits}
\begin{prop}
For a polynomial $p$ in $\mathbb{C}$, with $\deg(p)\geq 2$, the following holds:
\begin{align*}
\{M\in \mathcal{D}: \text{the eigenvalues of } M \text{are in }K_p\}
\cup \{M\in \mathcal{ND}:\text{the eigenvalue of } M \text{ is in } {\rm{int}}(K_p)\}  \\
\subseteq  K_P \subseteq \{M\in \mathcal{M}(2,\mathbb{C}): \text{the eigenvalues of $M$ are in $K_p$} \}.
\end{align*}
\end{prop}
\begin{proof}
For a diagonalizable matrix
\[
M=Q
\left[ {\begin{array}{cc}
\lambda_1 & 0 \\
0 & \lambda_2 \\
\end{array} } \right]Q^{-1},
\] 
we have 
\[
 P^r(M)=Q
  \left[ {\begin{array}{cc}
   p^r(\lambda_1) & 0 \\
   0 & p^r(\lambda_2) \\
  \end{array} } \right]Q^{-1}.
\] 
Clearly, $M\in K_P$ if and only if $\lambda_1, \lambda_2\in K_p$. 
Thus, the orbit of a diagonalizable matrix is bounded under iterations of $P$ if and only if both of its eigenvalues are in $K_p$. Now let
\[
 M=Q
  \left[ {\begin{array}{cc}
   \lambda & 1 \\
   0 & \lambda \\
  \end{array} } \right]Q^{-1},
\] 
a non-diagonalizable matrix, then
\[
 P^r(M)=Q
  \left[ {\begin{array}{cc}
   p^r(\lambda) & (p^r)'(\lambda) \\
   0 & p^r(\lambda) \\
  \end{array} } \right]Q^{-1}.
\]
Thus $\lambda$ must be in $K_p$, in case  $M$ has a bounded orbit. In particular, if $\lambda\in \rm{int}(K_p)$, then $\lambda$ belongs to some Fatou component of the polynomial $p$ in $\mathbb{C}$. By the classification theorem (\cite{MNTU}), several cases may arise.  

\medskip 
\no 
{\it{Case $1$}:} Let $\lambda\in U$ where $U$ is a Fatou component  of $p$ containing an attracting periodic point $\lambda_0$ of order $k$. Then for all $z\in U$,
\[
p^{rk}(z)\rightarrow \lambda_0
\]
as $r\rightarrow \infty$. Further, there exists a neighborhood $U_\lambda\subset \subset U$ of $\lambda$ such that the sequence $\{p^{rk}\}$ converges uniformly to $\lambda_0$ as $r\rightarrow \infty$. Thus, there exists $M_\lambda>0$ such that $\lvert p^r(\lambda)\rvert < M_\lambda$ for all $r\geq 1$.  

\medskip 
\no 
Note that ${(p^{rk})}'\rightarrow 0$  uniformly on $U_\lambda$ as $r\rightarrow \infty$.
Now for any $rk<s<(r+1)k$, 
\[
{(p^{s})}'(\lambda)=(p^{(s-rk)})'(p^{rk}(\lambda))(p^{rk})'(\lambda)
\]
and since $p^{rk}(\lambda)\rightarrow \lambda_0$, there exists a uniform bound $\tilde{M}_\lambda>0$ such that $\lvert(p^{(s-rk)})'(p^{rk}(\lambda))\rvert<\tilde{M}_\lambda$ for each $r\geq 1$. Thus, combining these two facts, it follows that $(p^s)'(\lambda)\rightarrow 0$ as $s\rightarrow \infty$. Therefore, in this case, $M\in K_P$.

\medskip
\no
{\it{Case $2$}:} Let $U$ be the immediate basin for some petal $\mathcal{P}$ of a parabolic periodic point $\lambda_0$ of period $k\geq 1$. Let $\lambda\in U$, then there exists a small neighborhood $U_\lambda$ of $\lambda$ in $U$ such that $p^{rk}(z)\in \mathcal{P}$ for all $z\in U_\lambda$ and for all $r\geq r_0$. Further, the sequence $\{p^{rk}\}$ converges to $\lambda_0$ uniformly on $U_\lambda$. So $\{(p^{rk})'\}$ converges to $0$ uniformly on $U_\lambda$. Thus, implementing a similar set of arguments as in the previous case, it follows that  $M\in K_P$.

\medskip
\no
{\it{Case $3$}:} Let $U$ be a Siegel disk. Then $U$ is conformally equivalent to the disk $\mathbb{D}$ and there exists a conformal map $h$ such that
\[
h\circ p^k \circ h^{-1}(z)=e^{i\theta} z
\]          
in $\mathbb{D}$ for some $\theta\in [0,2\pi)$. Thus, 
\[
(p^{rk})'(\lambda)=e^{ir\theta} \frac{h'(\lambda)}{h'(p^{rk}(\lambda))}.
\] 
If for some fixed $\lambda\in U$, $p^{rk}(\lambda)\rightarrow \partial U$ as $r\rightarrow \infty$, then $h(p^{rk}(\lambda))\rightarrow \partial \mathbb{D}$ as $r\rightarrow \infty$. Now since $h(p^{rk}(\lambda))=e^{ir\theta}h(\lambda)$ for all $r\geq 1$, we have $\overline{\{p^{rk}(\lambda)\}}\subset \subset U$. Thus for a fixed $\lambda\in U$, there exists $M_\lambda>1$ such that
\[
\lvert (p^{rk})'(\lambda) \rvert < M_\lambda
\]
for all $r\geq 1$. Now using similar arguments as in the previous cases, we can conclude that there exists a uniform bound $\tilde{M}_\lambda>1$ such that $\lvert p^r(\lambda)\rvert, \lvert(p^r)'(\lambda)\rvert<\tilde{M}_\lambda$ for each $r\geq 1$. Thus in this case, $M\in K_P$. This completes the proof of the first inclusion.

To prove the other inclusion, note that if any of the eigenvalues of a matrix $M$ is in $F_\infty(p)$, then $\lVert P^n(M)\rVert\rightarrow \infty$ as $n\rightarrow \infty$, since the norm of a matrix dominates its spectral radius.
\end{proof}

\begin{rem}
Consider the map $P:M\mapsto M^2$ in $\mathcal{M}(2,\mathbb{C})$. Then for $\lambda\in \partial K_p=\{z\in \mathbb{C}:\lvert z\rvert=1\}$ where $p(z)=z^2$, the orbit of the matrix $[\lambda]$ under the iterations of $P$ is not bounded.  On the contrary,  a parabolic fixed point $\lambda$ of any arbitrary polynomial $p$ always lie in $J_p$ and the matrix $[\lambda]$ has always bounded orbit under the iterations of the corresponding matrix-valued polynomial $P$.    	
\end{rem}
Note that $\overline{K_P}=\{M: \text{The eigenvalues of $M$ lie in $K_p$}\}$. 
\begin{prop}\label{Fatou}
Let $p$ be a polynomial in $\mathbb{C}$ with $\deg(p)\geq 2$ and $P$ be the corresponding matrix-valued polynomial in $\mathcal{M}(2,\mathbb{C})$, then
$$
\partial \overline{K_P}=\{Q[\lambda_1,\lambda_2]Q^{-1}: (\lambda_1,\lambda_2)\in (K_p\times J_p)\cup (J_p\times K_p), Q\in {\rm{GL}}(2,\mathbb{C})\}.
$$
\end{prop}
\begin{proof}
Let $M\in \partial \overline{K_P}$, this implies $M\notin \rm{int}(\overline{K_P})$ and thus $M\notin \rm{int}({K_P})$. Let $M$ be a diagonalizable matrix with both of the eigenvalues
 $\lambda_1, \lambda_2\in \rm{int}(K_p)$, then there exists a small neighborhood $\mathcal{N}_M$ of $M$ such that the eigenvalues of the matrices which lie in $\mathcal{N}_M$ belong to ${\rm{int}}(K_p)$. Thus any $W \in \mathcal{N}_M $ has a bounded orbit under iterations of $P$. Applying a similar set of arguments for the non-diagonalizable matrices with eigenvalues in $\rm{int}(K_p)$,  we conclude that
$$
\partial \overline{K_P}\subseteq \{Q[\lambda_1,\lambda_2]Q^{-1}: (\lambda_1,\lambda_2)\in (K_p\times J_p)\cup (J_p\times K_p), Q\in {\rm{GL}}(2,\mathbb{C})\}.
$$
The containment in the other direction is immediate. Hence the proof follows.
\end{proof}

\begin{rem}
For a polynomial $p$ in $\mathbb{C}$, the filled Julia set $K_p$ is always closed but note that  $K_P$ might not be closed in general.
 
\end{rem}

\section{Julia sets of  matrix-valued polynomials}
\begin{prop}
Let $P$ be the matrix-valued polynomial corresponding to a polynomial $p$ in $\mathbb{C}$ with $\deg(p)\geq 2$. Then a matrix $M\in F_P$ if and only if $QMQ^{-1}\in F_P$ for all $Q\in {\rm{GL}}(2,\mathbb{C})$. Hence, a matrix $M\in J_P$ if and only if $QMQ^{-1}\in J_P$ for all $Q\in {\rm{GL}}(2,\mathbb{C})$.
\end{prop}
\begin{proof}
Let $M\in F_P$ and $\mathcal{N}_M$ be a neighborhood of $M$ where, without loss of generality, the sequence $\{P^n\}$ converges uniformly to a holomorphic function or the sequence $\{P^n\}$ diverges uniformly to infinity. Now consider the matrix $QMQ^{-1}$ for some $Q\in {\rm{GL}}(2,\mathbb{C})$ and consider the open neighborhood $\mathcal{N}_M^Q=\{QXQ^{-1}: X\in \mathcal{N}_M\}$ of $QMQ^{-1}$. It is easy to see that if the sequence $\{P^n\}$ converges uniformly to a holomorphic function $R$ or diverges uniformly to infinity in $\mathcal{N}_M$, then the sequence $\{P^n\}$ converges uniformly to the holomorphic function $QRQ^{-1}$ or diverges uniformly to infinity in $\mathcal{N}_M^Q$, respectively. Thus $QMQ^{-1}\in F_P$. This finishes the proof. 
\end{proof}
\subsection{Proof of Theorem \ref{Julia}}
If any one of the eigenvalues of a matrix $M\in \mathcal{M}(2,\mathbb{C})$ is in $F_\infty(p)$, then there exists a neighborhood $\mathcal{N}_M$ of $M$ such that the sequence $\{P^n\}$ diverges uniformly  in $\mathcal{N}_M$. Thus,  $J_P\subset \overline{K_P}$. 
 
During the course of the proof, we assume that if we start with a matrix $M_0\in \mathcal{M}(2, \mathbb{C})$ 
with eigenvalues $\lambda_{M_0}, \mu_{M_0} \in F_p$, then the corresponding Fatou components containing these eigenvalues are remained fixed by the polynomial $p$ in $\mathbb{C}$.  If this is not the case, then there exists $k\geq 2$ such that both the components are fixed by $p^k$. Then, instead of working with the sequence $\{P^j\}_{j\geq 1}$, we work with $\{P^{kj}\}_{j\geq 1}$.  As we shall see, this assumption is harmless since the sequence $\{P^j\}_{j\geq 1}$ is normal if and only if the sequence $\{P^{kj}\}_{j\geq 1}$ is normal for any $k\geq 1$.

\medskip
\no
{\it{Case 1:}} 
Let \[
M_0=
Q_0\left[\begin{array}{cc}
\lambda_{M_0} & 0\\
0 & \mu_{M_0}
\end{array}
\right]Q_0^{-1}
\]
such that $\lambda_{M_0}$ and $\mu_{M_0}$ are in  ${\rm{int}}(K_p)$ and $\lambda_{M_0}\neq \mu_{M_0}$. Further, we assume that each of these components is either attracting or parabolic
and the points belonging to the Fatou components mentioned above are attracted towards the fixed points $\lambda$ and $\mu$, respectively (note that $\lambda$ and $\mu$ can be the same!).

Since, $\lambda_{M_0}\neq \mu_{M_0}$, there exists a small neighborhood $\mathcal{N}_{M_0}$ of $M_0$ such that each $M\in \mathcal{N}_{M_0}$ is of the form
 \[
Q_M\left[\begin{array}{cc}
\lambda_{M} & 0\\
0 & \mu_{M}
\end{array}
\right]Q_M^{-1}
\] 
with $\lambda_M\neq \mu_M$ and $\lambda_M, \mu_M$ are in the same component of ${\rm{int}(K_p)}$ as of $\lambda_{M_0}, \mu_{M_0}$, respectively. 

Using Theorem 3.1 in \cite{AK} in present set-up, the family $\mathcal{F}={\{P^n\}}_{n\geq 1}$ is not normal in $\mathcal{N}_{M_0}$ if and only if there exist a compact set $K_0\subset \subset \mathcal{N}_{M_0}$, a sequence of matrices $\{M_{j_k}\}$, a sequence of real numbers $\{\rho_{j_k}\}$ with $\rho_{j_k}>0$ and $\rho_{j_k}\rightarrow 0^+$ and a sequence of Euclidean unit vectors $\{\xi_{j_k}\}$ in $\mathbb{C}^4$ such that the sequence of entire functions
\[
g_k(\zeta)=P^{j_k}(M_{j_k}+\rho_{j_k} \xi_{j_k} \zeta), \ \ \zeta\in \mathbb{C}
\] 
converges uniformly on compact subset of $\mathbb{C}$ to a non-constant entire function $g$ as $k\rightarrow \infty$. Assume that the sequence $\{P^n\}$ is not normal in $\mathcal{N}_{M_0}$ and $M_{j_k}\rightarrow M\in \mathcal{N}_{M_0}$ as $k\rightarrow \infty$, then for all $\zeta\in \mathbb{C}$,
\begin{equation}\label{covg}
M_{j_k}+\rho_{j_k} \xi_{j_k} \zeta \rightarrow M
\end{equation}
as $k\rightarrow \infty$. Let 
\[
M=Q_M\left[\begin{array}{cc}
\lambda_{M} & 0\\
0 & \mu_{M}
\end{array}
\right]Q_{M}^{-1}
\]
with $\lambda_M \neq \mu_M$.

Fix any $\zeta\in \mathbb{C}$. Now we can choose a sequence of matrices $\{Q_{j_k}^\zeta\}$ such that
\begin{eqnarray}\label{matrix}
M_{j_k}+\rho_{j_k} \xi_{j_k} \zeta &=& Q_{j_k}^\zeta \left[\begin{array}{cc}
\lambda_{j_k}^\zeta & 0\\
0 & \mu_{j_k}^\zeta 
\end{array}
\right]{Q_{j_k}^\zeta}^{-1} \nonumber \\
&=&\left[\begin{array}{cc}
\lambda_{j_k}^\zeta+\frac{b_{j_k}^\zeta c_{j_k}^\zeta}{{\rm{det}}Q_{j_k}^\zeta}(\lambda_{j_k}^\zeta-\mu_{j_k}^\zeta) & \frac{a_{j_k}^\zeta b_{j_k}^\zeta}{{\rm{det}}Q_{j_k}^\zeta}(\lambda_{j_k}^\zeta- \mu_{j_k}^\zeta) \\
\frac{c_{j_k}^\zeta d_{j_k}^\zeta}{{\rm{det}}Q_{j_k}^\zeta}(\lambda_{j_k}^\zeta-\mu_{j_k}^\zeta) & \mu_{j_k}^\zeta+\frac{b_{j_k}^\zeta c_{j_k}^\zeta}{{\rm{det}}Q_{j_k}^\zeta}(\mu_{j_k}^\zeta-\lambda_{j_k}^\zeta)
\end{array}
\right]
\end{eqnarray}
where 
\[
Q_{j_k}^\zeta=\left[\begin{array}{cc}
a_{j_k}^\zeta & b_{j_k}^\zeta\\
c_{j_k}^\zeta & d_{j_k}^\zeta 
\end{array}
\right].
\]
Note that for each $\zeta\in \mathbb{C}$,   $\lambda_{j_k}^\zeta\neq \mu_{j_k}^\zeta$ for all  $k\geq k_\zeta>1$ sufficiently large. Further, $\lambda_{j_k}^\zeta\rightarrow \lambda_M$, $\mu_{j_k}^\zeta\rightarrow \mu_M$ as $k\rightarrow \infty$. Also using (\ref{covg}) and (\ref{matrix}),  it follows that there exists a constant $L>1$, independent of $\zeta$, such that 
\begin{equation}\label{bound1}
\left \lvert \frac{b_{j_k}^\zeta c_{j_k}^\zeta}{{\rm{det}}Q_{j_k}^\zeta}\right\rvert, 
\left \lvert \frac{a_{j_k}^\zeta b_{j_k}^\zeta}{{\rm{det}}Q_{j_k}^\zeta} \right \rvert,
\left \lvert \frac{c_{j_k}^\zeta d_{j_k}^\zeta}{{\rm{det}}Q_{j_k}^\zeta} \right \rvert < L
\end{equation}
for all $k\geq k_\zeta$. 
 
Now
\[
P^{j_k}(M_{j_k}+\rho_{j_k} \xi_{j_k} \zeta)=Q_{j_k}^\zeta \left[\begin{array}{cc}
p^{j_k}(\lambda_{j_k}^\zeta) & 0\\
0 & p^{j_k}(\mu_{j_k}^\zeta)
\end{array}
\right]{Q_{j_k}^\zeta }^{-1}
\]
and $p^{j_k}(\lambda_{j_k}^\zeta) \rightarrow \lambda$ and $p^{j_k}(\mu_{j_k}^\zeta)\rightarrow \mu$  as $k\rightarrow \infty$. This happens because for a fixed $\zeta$, each of $M_{j_k}+\rho_{j_k} \xi_{j_k} \zeta$ lie in a compact subset of $\mathcal{N}_{M_0}$ for sufficiently large $k$. Therefore, there are compact sets in the corresponding Fatou components containing $\lambda_{j_k}^\zeta$ and $\mu_{j_k}^\zeta$ for all $k$'s sufficiently large. Since,
\[
g(\zeta)=\lim_{k\rightarrow \infty} Q_{j_k}^\zeta \left[\begin{array}{cc}
p^{j_k}(\lambda_{j_k}^\zeta) & 0\\
0 & p^{j_k}(\mu_{j_k}^\zeta)
\end{array}
\right]{Q_{j_k}^\zeta}^{-1}
\]
\[
=\lim_{k\rightarrow \infty}\left[\begin{array}{cc}
p^{j_k}(\lambda_{j_k}^\zeta)+\frac{b_{j_k}^\zeta c_{j_k}^\zeta}{{\rm{det}}Q_{j_k}^\zeta}(p^{j_k}(\lambda_{j_k}^\zeta)-p^{j_k}(\mu_{j_k}^\zeta)) & \frac{a_{j_k}^\zeta b_{j_k}^\zeta}{{\rm{det}}Q_{j_k}^\zeta}(p^{j_k}(\lambda_{j_k}^\zeta)- p^{j_k}(\mu_{j_k}^\zeta)) \\
\frac{c_{j_k}^\zeta d_{j_k}^\zeta}{{\rm{det}}Q_{j_k}^\zeta}(p^{j_k}(\lambda_{j_k}^\zeta)-p^{j_k}(\mu_{j_k}^\zeta)) & p^{j_k}(\mu_{j_k}^\zeta)+\frac{b_{j_k}^\zeta c_{j_k}^\zeta}{{\rm{det}}Q_{j_k}^\zeta}(p^{j_k}(\mu_{j_k}^\zeta)-p^{j_k}(\lambda_{j_k}^\zeta))
\end{array}
\right],
\]
by (\ref{bound1}), it follows that there exists $K>1$ such that $\lVert g(\zeta)\rVert < K$ for all $\zeta\in \mathbb{C}$. Thus $g$ is identically constant in $\mathbb{C}$ which in turn gives that $\{P^n\}$ is a normal family in $\mathcal{N}_{M_0}$. 

\medskip
\no 
{\it{Case 2:}} As in the previous case, let us start with a matrix  $M_0$
such that its eigenvalues $\lambda_{M_0}$ and $\mu_{M_0}$  lie in two different components $U$ and $V$ of ${\rm{int}}(K_p)$. In this case, we assume one of these components ($U$, say) is either parabolic or an attracting one and the other component ($V$) a Siegel disk, both are fixed by the polynomial $p$. Clearly, $\lambda_{M_0}\neq \mu_{M_0}$ and we can choose a small neighborhood $\mathcal{N}_{M_0}$ of $M_0$ as before such that for all matrices $N \in \mathcal{N}_{M_0}$, the eigenvalues $\lambda_N, \mu_N$ are in the same components of $\rm{int}(K_p)$ as of $\lambda_{M_0}$, $\mu_{M_0}$. 
 
Let $M_{j_k}\rightarrow M$ as $k\rightarrow \infty$ (choice of $M_{j_k}$'s follows the same recipe as in the previous case) and consequently,  for all $\zeta\in \mathbb{C}$,
\[
M_{j_k}+\rho_{j_k} \xi_{j_k} \zeta\rightarrow M=Q_M\left[\begin{array}{cc}
\lambda_M & 0\\
0 & \mu_M
\end{array}
\right]Q_{M}^{-1}\in \mathcal{N}_{M_0}
\]   
as $k\rightarrow \infty$.
Clearly, $M$ can be diagonalized in such a way that $\lambda_M\in U$ and $\mu_M\in V$.
Using a same sort of arguments as in the previous case, it follows that 
$$
\lim_{k\rightarrow \infty}p^{j_k}(\lambda_{j_k}^\zeta)=\lim_{k\rightarrow \infty}p^{j_k}(\lambda_M)=\lambda_0
$$
and since $\lim_{k\rightarrow \infty} \mu_{j_k}^\zeta=\mu_M$ for all $\zeta\in \mathbb{C}$, there exists a subsequence 
$\{j_{k'}\}_{k'\geq 1}\subseteq \{j_k\}_{k\geq 1}$ (depending on $\zeta$), with $j_{k'} \rightarrow \infty$ as $k'\rightarrow \infty$, such that
$$
\lim_{k'\rightarrow \infty} p^{j_{k'}}(\mu_{j_{k'}}^\zeta)=\lim_{k'\rightarrow \infty} p^{j_{k'}}(\mu_M)=\mu_0
$$
for all $\zeta\in \mathbb{C}$. Since $\mu_0$ lies in the corresponding Siegel disk, $\lambda_0\neq \mu_0$.  Hence using a same sort of argument as in the Case $1$, we can establish that $\{P^n\}$ is a normal family in $\mathcal{N}_{M_0}$.

\medskip
\no 
{\it{Case 3:}}
Using the same notations as before, let $\lambda_{M_0}$ and $\mu_{M_0}$ both are in the Siegel disk (they might be in the same Siegel disk).  

\no
{\it{Subcase 1:}}
Let $\lambda_{M_0}\neq \mu_{M_0}$. As before, let
\[
\lim_{k\rightarrow \infty} \lambda_{j_k}^\zeta=\lambda_M
\text{ and } 
\lim_{k\rightarrow \infty} \mu_{j_k}^\zeta=\mu_M
\]
with $\lambda_M \neq \mu_M$.
Therefore there exists a subsequence $\{j_{k'}\}_{k'\geq 1}\subseteq \{j_k\}_{k\geq 1}$ (depending on $\zeta$) such that
$$
\lim_{k'\rightarrow \infty}p^{j_{k'}}(\lambda_{j_{k'}}^\zeta)
=\lim_{k'\rightarrow \infty} p^{j_{k'}}(\lambda_M)=\lambda_0
\text{ and } 
\lim_{k'\rightarrow \infty}p^{j_{k'}}(\mu_{j_{k'}}^\zeta)
=\lim_{k'\rightarrow \infty} p^{j_{k'}}(\mu_M)=\mu_0. 
$$
Since $\lambda_M \neq \mu_M$, $p^{j_{k'}}(\lambda_M)$ and $p^{j_{k'}}(\mu_M)$ can not come arbitrarily close. Thus $\lambda_0 \neq \mu_0$. Now using a same sort of argument as in the previous cases, we have that the family $\{P^n\}$ is normal in $\mathcal{N}_{M_0}$.

\no 
{\it{Subcase 2:}} Next we assume that $\lambda_{M_0}=\mu_{M_0}$. Following the same notations as before,  for all $\zeta \in \mathbb{C}$,
\[
M_{j_k}+\rho_{j_k}\xi_{j_k}\zeta\rightarrow M
\] 
as $k\rightarrow \infty$. 
Let $\mathbf{\lambda_M=\mu_M}$. 

Now let $\zeta\in \mathbb{C}$ is such that $(M_{j_k}+\rho_{j_k}\xi_{j_k}\zeta)$ is diagonalizable for infinitely many $k\in \mathbb{N}$ , then
\[
M_{j_k}+\rho_{j_k} \xi_{j_k} \zeta =Q_{j_k,\zeta}\left[\begin{array}{cc}
\lambda_{j_k,\zeta} & 0\\
0 & \mu_{j_k,\zeta}
\end{array}
\right]Q_{j_k,\zeta}^{-1}
\]
with $\lambda_{j_k,\zeta}, \mu_{j_k,\zeta}\rightarrow \lambda_M=\mu_M$ as $k\rightarrow \infty$.  As described in previous case, there exists a subsequence $\{j_{k'}\}\subseteq \{j_k\}$ such that
\[
\lambda=\lim_{k'\rightarrow \infty}p^{j_{k'}}(\lambda_{j_{k'},\zeta})=\lim_{k'\rightarrow \infty} p^{j_{k'}} (\mu_{j_{k'},\zeta})=\lim_{k'\rightarrow \infty} p^{j_{k'}}(\lambda_M)=\lambda.
\] 
Therefore,
\[
g(\zeta)=\lim_{k'\rightarrow \infty}P^{j_{k'}}(M_{j_{k'}}+\rho_{j_{k'}} \xi_{j_{k'}} \zeta)=\lambda {\rm{Id}}.
\]
Now let $\zeta\in \mathbb{C}$ be such that for which  $M_{j_k}+\rho_{j_k} \xi_{j_k} \zeta$ is non-diagonalizable for all $k\geq k_\zeta$.
For such a $\zeta\in \mathbb{C}$,
\[
M_{j_k}+\rho_{j_k} \xi_{j_k} \zeta =Q_{j_k}^\zeta\left[\begin{array}{cc}
\lambda_{j_k,\zeta}& A_{{j_k},\zeta}\\
0 & \lambda_{j_k,\zeta}
\end{array}
\right]{Q_{j_k}^\zeta}^{-1}.
\]
for all $k\geq k_\zeta$ with $\lambda_{j_k,\zeta}\rightarrow \lambda_M$ as $k\rightarrow \infty$. 
Now 
\[
Q_{j_k}^\zeta \left[\begin{array}{cc}
\lambda_{j_k,\zeta}& A_{{j_k},\zeta}\\
0 & \lambda_{j_k,\zeta}
\end{array}
\right]{Q_{j_k}^\zeta}^{-1}
= \left[\begin{array}{cc}
\lambda_{j_k,\zeta}-\frac{a_{j_k}^\zeta c_{j_k}^\zeta}{\det Q_{j_k}^\zeta} A_{j_k,\zeta} & \frac{{(a_{j_k}^\zeta)}^2}{\det Q_{j_k}^\zeta} A_{{j_k},\zeta}\\
-\frac{{(c_{j_k}^\zeta)}^2}{\det Q_{j_k}^\zeta} A_{{j_k},\zeta} & \lambda_{j_k,\zeta}+\frac{a_{j_k}^\zeta c_{j_k}^\zeta}{\det Q_{j_k}^\zeta} A_{j_k,\zeta}
\end{array}
\right].
\]
Now since for each $\zeta\in \mathbb{C}$, $(M_{j_k}+\rho_{j_k} \xi_{j_k}\zeta)\rightarrow M$ and $\lambda_{j_k,\zeta}\rightarrow \lambda_M$ as $k\rightarrow \infty$ , it follows that for each $\zeta$, there exists $k_{\zeta}'$ and an $L>1$ (not depending on $\zeta$) such that
\begin{equation}
\left\lvert \frac{a_{j_k}^\zeta c_{j_k}^\zeta}{\det Q_{j_k}^\zeta} A_{j_k,\zeta}\right\rvert, \left\lvert\frac{{(a_{j_k}^\zeta)}^2}{\det Q_{j_k}^\zeta} A_{{j_k},\zeta}\right\rvert, \left\lvert\frac{{(c_{j_k}^\zeta)}^2}{\det Q_{j_k}^\zeta} A_{{j_k},\zeta}\right\rvert < L 
\end{equation}
for all $k\geq k_\zeta'$.  
Now 
\begin{eqnarray*}
g(\zeta)&=&\lim_{k\rightarrow \infty} P^{j_k} \left(M_{j_k}+\rho_{j_k}\xi_{j_k} \zeta\right)\\
&=& \lim_{k\rightarrow \infty} Q_{j_k}^\zeta\left[\begin{array}{cc}
p^{j_k}(\lambda_{j_k,\zeta})& A_{{j_k},\zeta}(p^{j_k})'(\lambda_{j_k,\zeta})\\
0 & p^{j_k}(\lambda_{j_k,\zeta})
\end{array}
\right]{Q_{j_k}^\zeta}^{-1}\\
&=& \lim_{k\rightarrow \infty}\left[\begin{array}{cc}
p^{j_k}(\lambda_{j_k,\zeta})-\frac{a_{j_k}^\zeta c_{j_k}^\zeta }{\det Q_{j_k}^\zeta} A_{j_k,\zeta}(p^{j_k})'(\lambda_{j_k,\zeta}) & \frac{{(a_{j_k}^\zeta)}^2}{\det Q_{j_k}^\zeta} A_{{j_k},\zeta}(p^{j_k})'(\lambda_{j_k,\zeta})\\
-\frac{{(c_{j_k}^\zeta)}^2}{\det Q_{j_k}^\zeta} A_{{j_k},\zeta}(p^{j_k})'(\lambda_{j_k,\zeta}) & p^{j_k}(\lambda_{j_k,\zeta})+\frac{a_{j_k}^\zeta c_{j_k}^\zeta}{\det Q_{j_k}^\zeta} A_{j_k,\zeta}(p^{j_k})'(\lambda_{j_k,\zeta})
\end{array}
\right].
\end{eqnarray*}
Since for each $\zeta\in \mathbb{C}$, $\lambda_{j_k,\zeta}\rightarrow \lambda_M$,  it follows that $\lvert p^{j_k}(\lambda_{j_k,\zeta})\rvert$  
and $\lvert( p^{j_k})'(\lambda_{j_k,\zeta})\rvert$ are uniformly bounded and the bound is independent of $\zeta$. Thus $g$ turns out to be a bounded entire functionon and hence $g$ becomes a constant function in $\mathbb{C}$, in fact $g\equiv \lambda {\rm{Id}}$ in $\mathbb{C}$, establishing the fact that the family $\{P^n\}$ is normal in $\mathcal{N}_{M_0}$. If $\lambda_M \neq \mu_M$, then as in previous cases, we can show that the family $\{P^n\}$ is normal in $\mathcal{N}_{M_0}$.

Now note that if we start with a matrix with identical eigenvalues such that they both are in the same attracting or in the same parabolic compoment of ${\rm{int}}(K_p)$, then using a similar treatment as in the  Case $3$, we can record the following case.

\medskip 
\no 
{\it{Case 4:}}  For a matirx $M_0\in \mathcal{M}(2,\mathbb{C})$, with two identical eigenvalues sitting either in an attracting or in a parabolic component of ${\rm{int}}(K_p)$, there exists a neighborhood $\mathcal{N}_{M_0}$ containing $M_0$ such that  $\{P^n\}_{n\geq 1}$ forms a normal family therein.   
This finishes the proof.
\subsection{Some properties of Julia sets} 
Now we discuss some fundamental properties of the Julia sets and Fatou sets of the matrix-valued polynomial $P$ which are simply  inherited from the properties of Julia sets and Fatou sets of $p$.  
Let 
\[
J_1=\left \{ Q[\lambda_1, \lambda_2]Q^{-1}: \text{either }(\lambda_1, \lambda_2)\in J_p \times {\rm{int}}(K_p) \text{ or }(\lambda_1, \lambda_2)\in  {\rm{int}}(K_p) \times J_p \right\}
\]
and
\[
J_2=\left \{ Q[\lambda]Q^{-1}:\lambda\in J_p\right\}
\cup 
\left \{ Q[\lambda_1, \lambda_2]Q^{-1}:\lambda_1, \lambda_2\in J_p \right\}.
\]
Note that 
\[
J=J_1\cup J_2.
\]

\medskip 
\no 
The following  string of corollaries are immediate.
\begin{cor}
	$J_1$, $J_2$ both are completely invariant under $P$. Further, $J_2$ is closed but $J_1$ is neither closed nor open. Also, $J=J_1\cup J_2$, $J_1\cap J_2=\emptyset$ and $\partial J_1=J_2$.
\end{cor}

\begin{cor}
The sets $J_1$ and $J_2$ are both infinite upto conjugation. 
\end{cor}

\begin{cor}
There exists no wandering Fatou component of $P$.
\end{cor}
\begin{proof}
By Proposition \ref{Fatou}, if a matrix $M\in F_P$, then $QMQ^{-1}\in F_P$ for all $Q\in {\rm{GL}}(2,\mathbb{C})$. Further,  since ${\rm{GL}}(2,\mathbb{C})$ is connected, if a matrix $M\in C_P$ where $C_P\subseteq F_P$ is a component of $F_P$, then $QMQ^{-1}$ also belongs to $C_P$ for all $Q\in {\rm{GL}}(2,\mathbb{C})$.

Now one can construct continuous functions 
\[
\lambda_1: C_P\rightarrow \mathbb{R} \text{ and } \lambda_2: C_P\rightarrow \mathbb{R}
\]
such that $\lambda_1(M)$ and $\lambda_2(M)$ are two eigenvalues of $M\in C_P$. Also 
\begin{equation*}\label{subset}
\lambda_1(C_P)\subseteq C_p^1\subseteq F_p \text{ and } \lambda_2(C_P)\subseteq C_p^2\subseteq F_p
\end{equation*}
where $C_p^1$ and $C_p^2$ are two components of $F_p$. Since there exists no wandering Fatou component of $p$, there exists $n_0\in \mathbb{N}$ such that 
\begin{equation*}
P^{n_0}(C_p^1)\subseteq C_p^1 \text{ and } P^{n_0}(C_p^2)\subseteq C_p^2
\end{equation*}
which in turn gives that 
\[
P^{n_0}(C_P)\subseteq C_P.
\]
  
\end{proof}
\no 
The following corollary is a generalization of Theorem 1 in \cite{Be} in our set-up. 
\begin{cor}
Let $J_P$  and $J_Q$ be the Julia sets of the matrix-valued polynomials  $P$ and $Q$, respectively and $P\circ Q=Q\circ P$. Then $J_P=J_Q$. Conversely if $J_P=J_Q$, then there exists a matrix-valued linear map  $\Sigma$ in $\mathcal{M}(2,\mathbb{C})$ such that $P\circ Q= \Sigma \circ Q \circ P$ where $\Sigma: M\mapsto aM+b$ with $\lvert a\rvert=1$. 
\end{cor}


\begin{cor}
The Julia set $J_P$ has always empty interior. Further, $J_P$ is never the whole set $\mathbb{C}^4$.
\end{cor}

\begin{cor}
The sets $J_P$ and $F_P$ are completely invariant under $P$.	
\end{cor}

\begin{prop}
Let P be a matrix-valued polynomial as above and $P(\mathcal{C})=\mathcal{C}$ where $\mathcal{C}=\{M\in \mathcal{M}(2,\mathbb{C}):\rho(M)=1\}$. Then, $P(M)=\alpha M^d$ for some $\alpha$ with $\lvert \alpha \rvert=1$.
\end{prop}

\begin{proof}
Let 
$$\mathcal{B}=\{M\in \mathcal{M}(2,\mathbb{C}):\rho(M)<1\} \text{ and } \mathcal{B^*}=\{M\in \mathcal{M}(2,\mathbb{C}):\rho(M)>1\}.
$$
Since $\mathcal{B}$ is connected, $P$ maps $\mathcal{B}$  either into $\mathcal{B}$ or into $\mathcal{B}^*$.  If $P(\mathcal{B})\subseteq \mathcal{B}^*$, then it contradicts the fact that $P(\mathcal{C})=\mathcal{C}$. Thus $P(\mathcal{B})\subseteq \mathcal{B}$. Consequently, we get 
$$p(\{z\in \mathbb{C}:\lvert z\rvert=1\})=\{z\in \mathbb{C}:\lvert z\rvert=1\}
$$
which implies $p(z)=\alpha z^d$ with $\lvert \alpha\rvert=1$ (by applying Theorem 1.3.1 in \cite{Be1}). This completes the proof.
\end{proof}

\no
A matrix $M$ is periodic for $P$ if $P^n(M)=M$ for some $n\geq 1$. Therefore a matrix 
  \[
 M=Q[\lambda_1, \lambda_2]Q^{-1}.
 \] is periodic for $P$ if and only if $\lambda_1$ and $\lambda_2$ both are periodic for the polynomial $p$ in $\mathbb{C}$. Similarly, 
 \[
 M=Q[\lambda]Q^{-1}
 \]
  is periodic for $P$ if and only if $p^n(\lambda)=\lambda$  and ${(p^n)}'(\lambda)=1$ for some $n\in \mathbb{N}$ i.e. $\lambda$ is a parabolic point for $p^n$. This implies $\lambda\in J_p$. Therefore, we have 
 \begin{prop}
 If a  non-diagonalizable $2\times 2$ matrix $M\in \mathcal{M}(2,\mathbb{C})$ is periodic, then $M\in J_P$.
 \end{prop}

 \section{Green functions and  B\"{o}ttcher functions}
\subsection{Proof of the theorem \ref{Green fun}} To show the existence of the dynamical Green functions of $P$, we first consider the matrices which are diagonalizable.
Note that the sequences
$$
\frac{1}{d^n}\log \lVert P^n(\Lambda)\rVert \text{ and } \frac{1}{d^n}\log \lVert P^n(Q\Lambda Q^{-1})\rVert
$$
if converge, converge simultaneously to the same matrix  where 
$
\Lambda=[\lambda_1, \lambda_2]
$
and
\begin{equation}\label{Green convg}
\frac{1}{d^n}\log \lVert P^n(\Lambda)\rVert=\frac{1}{d^n} \log \left(\max \{\lvert p^n(\lambda_1)\rvert, \lvert p^n(\lambda_2)\rvert\}\right).
\end{equation}
For any two complex numbers  $\lambda_1, \lambda_2\in F_p$, either $G_p(\lambda_1)\geq G_p(\lambda_2)$ or  $G_p(\lambda_1)<G_p(\lambda_2)$ where $G_p$ is the  Green function of $K_p$ with logarithmic pole at infinity (see, \cite{MNTU}). Therefore, the sequence (\ref{Green convg}) converges 
and 
\begin{equation}\label{max}
G(\Lambda)=\max \{G_p(\lambda_1), G_p(\lambda_2)\}.
\end{equation}
\no
Now there exists, $R>1$ large enough such that  $M_1<G_p(\lambda)-\log \lvert \lambda \rvert<M_2$ for $\lvert\lambda\rvert > R$ with $M_1, M_2 \in \mathbb{R}$.

\medskip 
\no 
{\it{Case 1:}}
Now let for a given diagonalizable matrix $M$,  $R<\rho(M)=\lvert\lambda_1\rvert \geq  \lvert \lambda_2\rvert \leq R$ and  $G_p(\lambda_1)\geq G_p(\lambda_2)$. Therefore 
\[
 M_1< G_p(\lambda_1)-\log \lvert \lambda_1\rvert = G(M)-\log \rho(M) < M_2. 
\]
Now let us consider the  same assumption as above except the fact that $G_p(\lambda_1)\geq G_p(\lambda_2)$, i.e,  we consider $G_P(\lambda_1)< G_p(\lambda_2)$. Then
\[
M_1<G_p(\lambda_1)-\log\lvert\lambda_1\rvert < G_p(\lambda_2)-\log\lvert\lambda_1\rvert=G(M)-\log \rho(M)< G_p(\lambda_2)< M
\]
where $M=\sup\{G_p(\lambda):\lvert\lambda\rvert=R\}$.

\medskip 
\no 
{\it{Case 2:}}
Let $\lvert\lambda_1\rvert \geq \lvert \lambda_2\rvert>R$ and $G_p(\lambda_1)\leq G_p(\lambda_2)$. Then 
\[
M_1<G_p(\lambda_1)-\log \lvert \lambda_1\rvert<G_p(\lambda_2)-\log \lvert \lambda_1\rvert < G_p(\lambda_2)-\log \lvert \lambda_2\rvert <M_2,
\]
which in turn gives
\[
M_1<G(M)-\log \rho(M)<M_2.
\]

\medskip 
\no 
If  $G_p(\lambda_1)>G_p(\lambda_2)$, then clearly, 
\[
M_1<G(M)-\log \rho(M)<M_2.
\]
Thus 
$$
G(M)=\log \rho(M)+O(1)
$$
in $\{\rho(M)>R\}$. 
 
Now we consider a non-diagonalizable matrix $\Lambda=[\lambda]$
and  consider the following sequence of matrices.
\[
\frac{1}{d^n}\log \lVert P^n(\Lambda)\rVert=\frac{1}{d^n} \log (\max \{\lvert p^n(\lambda)\rvert, \lvert (p^n)'(\lambda)\rvert\}).
\]
{\it{Case 1:}} Let $\lambda\in {\rm{int}}(K_p)$, then there exists an $L>1$ such that $\lvert (p^n)'(\lambda)\rvert<L^n$ since ${\rm{int}}(K_p)$ is invariant under $p$. Thus
\[
G(\Lambda)=\lim_{n\rightarrow \infty}\frac{1}{d^n}\log \lVert P^n(\Lambda)\rVert=0.
\] 

\no
{\it{Case 2:}} Let $\lambda\in I_p$. Further assume that $\varphi_p$, defined on $\{z:\lvert z \rvert>R\}$, is the B\"{o}ttcher coordinate of $p$ satisfying
\[
\phi_p(p^n(z))={(\phi_p(z))}^{d^n}
\]
for all $z$ with $\lvert z\rvert>R$ and for all $n\geq 1$. Now note that 
\begin{equation*}
\phi_p'(p^n(z))(p^n)'(z)=d^n {(\phi_p(z))}^{d^n-1}\phi_p'(z).
\end{equation*} 
Thus 
\begin{equation}
\frac{(p^n)'(z)}{d^n A_z p^n(z)}=\frac{\phi(p^n(z))}{p^n(z)}.\frac{1}{\phi'(p^n(z))}
\end{equation}
where $A_z={(\phi(z))}^{-1}\phi'(z)\neq 0$. Further, 
\[
\lim_{n\rightarrow \infty} \frac{\phi(p^n(z))}{p^n(z)}, \lim_{n\rightarrow \infty} \frac{1}{\phi'(p^n(z))} \neq 0.
\]
Thus 
\[
\lim_{n\rightarrow \infty}\frac{(p^n)'(z)}{p^n(z)d^n A_z} \neq 0
\]
and consequently, we get
\begin{equation}\label{lim}
\lim_{n\rightarrow \infty} \frac{1}{d^n}\log \lvert (p^n)'(z)\rvert=\lim_{n\rightarrow \infty} \frac{1}{d^n}\log \lvert (p^n)(z)\rvert.
\end{equation}
For $\lambda\in I_p$, there exists $n_0>1$ such that $p^{n_0}(\lambda)\in \{\lvert z\rvert>R\}$. Now 
\begin{eqnarray*}
\frac{1}{d^n}\lVert P^n(\Lambda)\rVert&=&\frac{1}{d^n}\log \left\lVert P^n 
\left[
\begin{array}{cc}
\lambda & 1 \\
0& \lambda
\end{array}
\right]\right\rVert\\
&=& \frac{d^{-n_0}}{d^{n-n_0}}\log \left\lVert P^{n-n_0} 
\left[
\begin{array}{cc}
p^{n_0}(\lambda) & (p^{n_0})'(\lambda) \\
0& p^{n_0}(\lambda)
\end{array}
\right]\right\rVert
\end{eqnarray*}
and this is equal to either
\[
\frac{d^{-n_0}}{d^{n-n_0}}\log \left\lVert Q_0 P^{n-n_0} 
\left[
\begin{array}{cc}
p^{n_0}(\lambda) & 1 \\
0& p^{n_0}(\lambda)
\end{array}
\right] Q_0^{-1}\right\rVert
\sim  \frac{d^{-n_0}}{d^{n-n_0}}\log \left\lVert P^{n-n_0} 
\left[
\begin{array}{cc}
p^{n_0}(\lambda) & 1 \\
0& p^{n_0}(\lambda)
\end{array}
\right]\right\rVert
\]
or
\[
\frac{d^{-n_0}}{d^{n-n_0}}\log \left\lVert Q_0 P^{n-n_0} 
\left[
\begin{array}{cc}
p^{n_0}(\lambda) & 0\\
0& p^{n_0}(\lambda)
\end{array}
\right] Q_0^{-1}\right\rVert
\sim  \frac{d^{-n_0}}{d^{n-n_0}}\log \left\lVert P^{n-n_0} 
\left[
\begin{array}{cc}
p^{n_0}(\lambda) & 0 \\
0& p^{n_0}(\lambda)
\end{array}
\right]\right\rVert
\]
for some $Q_0\in {\rm{GL}}(2,\mathbb{C})$.
In either of these cases, using (\ref{lim}) or (\ref{max}) respectively and implementing the functorial property of $G_p$, we finally get that 
\begin{equation}\label{nondia}
\frac{1}{d^n}\lVert P^n(\Lambda)\rVert
=d^{-n_0}\frac{1}{d^{n-n_0}}\lVert P^n(\Lambda)\rVert=d^{-n_0}G_p(p^{n_0}(\lambda))=G_p(\lambda).
\end{equation}
Therefore combining (\ref{max}) and (\ref{nondia}), we get that
\begin{equation}\label{Green des}
G(M)=\max\{G_p(\lambda_1), G_p(\lambda_2): \lambda_1, \lambda_2 \text{ are the eigenvalues of } M\}
\end{equation} 
and
\begin{equation}\label{Green}
G(P(M))=dG(M)
\end{equation}
for all $M\in \mathcal{M}(2,\mathbb{C})$ where $d$ is the degree of the polynomial $p$ in $\mathbb{C}$ which corresponds the matrix-valued polynomial $P$ in $\mathcal{M}(2,\mathbb{C})$. Since the Green function $G_p$ converges uniformly on compact subset of $\mathbb{C}$, that $G_n$ converges to $G$ uniformly on compacts follows from (\ref{Green des}).

Now choose $R>1$ sufficiently large such that $p(\{\lvert z\rvert>R\})\subset \{\lvert z\rvert>R\}\subset F_\infty(p)$ and we aim to show that $G$ is pluri-harmonic in $\Omega=\{\rho(M)>R\}$. Since $\rho(M)>R$, $\lVert M \rVert \geq \rho(M)>R$ and consequently, $\lVert P^n(M)\rVert\geq\rho(P^n(M))>R$. Consider following sequence of  positive pluri-harmonic functions 
$$
\frac{1}{d^n}\log \lVert P^n(M)\rVert
$$
in $\Omega$. Then by applying Harnack's theorem, we get that $G$ is a pluri-harmonic function in $\Omega$.\
 
Now for any $M\in \mathcal{M}(2,\mathbb{C})$ with atleast one of the eigenvalues in $F_\infty(p)$, there exists an open neighborhood $\mathcal{N}_M$ and $n_0>1$ such that $P^{n_0}(\mathcal{N}_M)\subseteq \Omega$. Thus $G\circ P^{n_0}$ is pluri-harmonic in $\mathcal{N}_{M_0}$. Using (\ref{Green}), we have that 
$$
G(P^{n_0}(M))=d^{n_0}G(M)
$$   
in $\mathcal{M}(2,\mathbb{C})$. Therefore, $G$ is pluri-harmonic in $\mathcal{N}_M$. Thus we prove that $G$ is pluri-harmonic in $F_\infty(P)$. Now by construction, $G$ vanishes in $\overline{K_P}=\mathcal{M}(2,\mathbb{C})\setminus F_\infty(P)$. Hence, $G$ is pluri-harmonic in $\rm{int}(\overline{K_P})$ and in turn, $G$ is pluri-harmonic in $F_P$. Clearly, $G$ is pluri-subharmonic in $\mathcal{M}(2,\mathbb{C})$.  

\subsection*{Proof of the theorem \ref{Bot}}
Let for $R>1$ large enough, $\varphi_p: \Omega_R\rightarrow \mathbb{C}$, where $\Omega_R=\{z\in \mathbb{C}:\lvert z \rvert > R\}$, defined as follow,
\[
z \mapsto bz+ b_0 + \frac{b_1}{z}+\cdots,
\]
be the B\"{o}ttcher function of the polynomial $p$ in $\mathbb{C}$. Now for each $n\geq 1$, define $$\Phi_n: \Omega=\{M\in \mathcal{M}(2,\mathbb{C}): \text{the eigenvalues of $M$ lie in } \Omega_R \}\rightarrow \mathcal{M}(2,\mathbb{C})$$ as follow, 
\[
\Phi_n(M)=bM+ b_0 + b_1 M^{-1}+\cdots+ b_n M^{-n}.
\]
Clearly, for each $n\geq 1$, $\Phi_n$ is a holomorphic function in $\Omega$. Define $\Phi: \Omega\rightarrow \mathcal{M}(2,\mathbb{C})$ as follow:
\begin{equation}\label{defPhi}
\Phi(M)=
\begin{cases}
 Q_M\left[ {\begin{array}{cc}
\varphi_p(\lambda_M) & 0 \\
0 & \varphi_p(\mu_M) \\
\end{array} } \right] Q_M^{-1} & \text{ if }M=Q_M[\lambda_M, \mu_M] Q_M^{-1} \\
\\
Q_M \left[ {\begin{array}{cc}
\varphi_p(\lambda_M) & \varphi_p'(\lambda_M) \\
0 & \varphi_p(\lambda_M) \\
\end{array} } \right]Q_M^{-1}  & \text{ if }M=Q_M[\lambda_M] Q_M^{-1} 
\end{cases}
\end{equation}
{\it{Claim:}} $\Phi_n$ converges to the function $\Phi$ uniformly on compacts in $\Omega$ as $n\rightarrow \infty$ and thus, $\Phi$ is holomorphic in $\Omega$.

\no
Let $M_0\in \Omega$ with  eigenvalues $\lambda_{M_0}$ and $\mu_{M_0}$ and assume that they are distinct.  Consider a small neighborhood $\mathcal{B}_{M_0}\subset \Omega$ of $M_0$. Let $M\in \mathcal{B}_{M_0}$, then
\[
M=
Q_M\left[\begin{array}{cc}
\lambda_{M} & 0\\
0 & \mu_{M}
\end{array}
\right]Q_M^{-1}
=
\left[\begin{array}{cc}
\lambda_{M}+\frac{b_M c_M}{\det Q_M}(\lambda_M-\mu_M) & \frac{a_M b_M}{\det Q_M}(\lambda_M-\mu_M)\\
\frac{c_M d_M}{\det Q_M}(\lambda_M-\mu_M) & \mu_M+\frac{b_M c_M}{\det Q_M}(\mu_M-\lambda_M)
\end{array}
\right]
\]
where 
\[ 
Q_M=
\left[\begin{array}{cc}
a_M & b_M\\
c_M & d_M
\end{array}\right]
\]
and clearly, $\lambda_M\neq \mu_M$.
Choosing $\mathcal{B}_{M_0}$ sufficiently small, we can choose $L_{M_0}>1$ such that 
\[
\left \lvert \frac{b_M c_M}{\det(Q_M)}\right \rvert, \left \lvert \frac{a_M b_M}{\det(Q_M)}\right \rvert, \left \lvert \frac{c_M d_M}{\det(Q_M)}\right \rvert < L_{M_0}
\] 
for all $M\in \mathcal{B}_{M_0}$.
Thus the modulus of each coordinate of $\Phi_n(M)$  is uniformly bounded by some constant for all $M\in \mathcal{B}_{M_0}$ and for all $n\geq 1$. Therefore, there exists a subsequence $\{\Phi_{n_k}\}$ of $\{\Phi_n\}$ which converges uniformly to a holomorphic function  $\Phi$ in $\mathcal{B}_{M_0}$. But note that for any $M\in \mathcal{B}_{M_0}$,  and  for any subsequence of $\{\Phi_n\}$, the limit function is always $\Phi$.
Hence $\{\Phi_n\}$ converges uniformly to the function $\Phi$ in $\mathcal{B}_{M_0}$. Now 
\begin{equation}\label{botgreen}
\varphi_p \circ p(z)={(p(z))}^d \text{ and } G_p(z)=\log \lvert \varphi_p(z)\rvert
\end{equation}
for all $z\in \Omega_R$ (see \cite{MNTU}). Then using (\ref{defPhi}) and (\ref{botgreen}), we get that
\begin{equation*} 
\Phi_P \circ P(M)= {(\Phi_P(M))}^d \text{ and } G(M)=\log \rho(\Phi_P(M))
\end{equation*}
for all $M\in \Omega$.
 
Now if $M_0\in \Omega$ is such that it has identical eigenvalues, then a similar set of arguments as above and the the same techniques as in the proof of the Theorem \ref{Julia} can be applied to prove that the sequence $\{\Phi_n\}$ converges uniformly to the holomorphic  function $\Phi$ on any compact set in $\Omega$ satisfying (\ref{botgreen}).  Hence the proof follows.

\end{document}